\documentclass[10pt]{article}
\usepackage{amsthm, amstext}
\usepackage{amsmath}
\usepackage{amsmath, amsthm, amssymb}

\usepackage{enumitem}
\usepackage{graphicx}
\usepackage{amsfonts}
\usepackage{latexsym, xypic}
\usepackage[all]{xy}
\xyoption{all}
\usepackage[colorlinks,linkcolor=blue,citecolor=blue]{hyperref}
\theoremstyle{plain}
\newtheorem{theorem}{Theorem}[section]
\newtheorem{lemma}[theorem]{Lemma}
\newtheorem{proposition}[theorem]{Proposition}

\theoremstyle{definition}
\newtheorem{definition}[theorem]{Definition}

\newtheorem{example}{\sc Example}
\theoremstyle{remark}

\theoremstyle{case}

\def\pi{positive implicative }

\date{}
\begin{document}

	\title{\bf {An Application of Maximality to Nilpotent and Finitely Generated  $L$-Subgroups of an $L$-Group}}
	\author{\textbf{Iffat Jahan$^1$ and Ananya Manas $^2$} \\\\ 
		$^{1}$
		Department of Mathematics, Ramjas College\\
		University of Delhi, Delhi, India \\
		ij.umar@yahoo.com \\\\
		$^{2}$Department of Mathematics, \\
		University of Delhi, Delhi, India \\
		anayamanas@gmail.com 
		\\    }
	\date{}
	\maketitle
	\begin{abstract}
		This paper is a continuation of the study on  maximal and Frattini $L$-subgroups of an $L$-group. The normality of the maximal $L$-subgroups of a nilpotent $L$-group is explored. Then, the concept of finitely generated $L$-subgroup is introduced and its relation with maximal condition on $L$-subgroups is established. Thereafter, several results pertaining to the notions of Frattini $L$-subgroup and finitely generated $L$-subgroups have been investigated.  
	\end{abstract}
	
\section{Introduction}

In 1965, Lofti A. Zadeh’s pioneering paper \cite{zadeh_fuzzy} laid the foundations of fuzzy logic in order to employ it to approximate reasoning. Ever since its inception, fuzzy sets and fuzzy logic have found applications in various fields-- information theory, linear programming, pattern recognition etc.  In 1971, Rosenfeld \cite{rosenfeld_fuzzy}
introduced the concept of fuzzy groups using the notion of fuzzy sets. Following this, a number of group theoretic concepts have been extended to the fuzzy setting, for example, the notion of set product of two fuzzy subgroups and the notion of normality of a fuzzy subgroup developed by Liu \cite{liu_op}. All of these studies have been carried out within the framework of fuzzy setting. Moreover, in these studies, the parent structure is considered to be an ordinary algebra (group, ring, semigroup etc.). Therefore, almost all the researchers have defined and studied the notion of a fuzzy subgroup of an ordinary group. This approach does not permit the transference of many properties of algebraic structures into the fuzzy setting. This shortcoming can be easily overcome if the parent structure is considered to be a fuzzy group instead of a classical group. Indeed, very few researchers such as Martinez \cite{martinez_fuzzy}, and Malik and Mordeson \cite{malik_pri}  have studied the properties of an $L$-subalgebra of an $L$-algebra. Also, Prajapati and Ajmal \cite{prajapati_max1,prajapati_max2} have defined and explored the notion of maximal $L$-ideals of a fuzzy ring. However, such studies have been very much inadequate in case of $L$-groups. Recently, various concepts of classical group theory have been extended to $L$-setting specially keeping in view their compatibility \cite{ajmal_gen,ajmal_nc,ajmal_nil,ajmal_nor, ajmal_sol}. Moreover, in \cite{jahan_max}, the authors have investigated the maximal and Frattini $L$-subgroups of an $L$-group. The notion of non-generators of an $L$-group has also been introduced and a relation between the Frattini $L$-subgroup and the set of non-generators of an $L$-group has been established. This work is a continuation of these topics.   

We begin our study in section 3, by demonstrating that every maximal $L$-subgroup of a nilpotent $L$-group $\mu$ having the same tip and tail as $\mu$ is normal in $\mu$. To establish this result, the concepts of nilpotent $L$-subgroups \cite{ajmal_nil}, normalizer of an $L$-subgroup of an $L$-group \cite{ajmal_nor} and normal closure of an $L$-subgroup of an $L$-group \cite{ajmal_nc} are used. Thus, this section demonstrates the compatibility of the above mentioned concepts. 

Then, in section 4, the maximal condition on $L$-subgroups is introduced. It is shown that an $L$-group satisfies maximal condition on $L$-subgroups if and only if every proper ascending chain of its $L$-subgroups is finite. Thereafter, the notion of finitely generated $L$-subgroup is established through $L$-points. A relation between the maximal condition on $L$-subgroups and finitely generated $L$-subgroups is also established, like their classical counterparts.

We begin section 5 by showing that if $L$ is an upper well ordered lattice and $\mu$ is a normal $L$-group, then the Frattini $L$-subgroup $\Phi(\mu)$ is normal in $\mu$. Thereafter, we prove that every finitely generated $L$-group, where $L$ is an upper well ordered lattice, has a maximal $L$-subgroup. Moreover, several relations between Frattini $L$-subgroups, finitely generated $L$-subgroups and set product of $L$-subgroups are established.   

\section{Preliminaries}

Throughout this paper, $L = \langle L, \leq, \vee, \wedge  \rangle$ denotes a completely distributive lattice where '$\leq$' denotes the partial ordering on $L$ and '$\vee$'and '$\wedge$' denote, respectively, the join (supremum) and meet (infimum) of the elements of $L$. Moreover, the maximal and minimal elements of $L$ will be denoted by $1$ and $0$, respectively. The concept of completely distributive lattice can be found in any standard text on the subject \cite{gratzer_lattices}. 

The notion of a fuzzy subset of a set was introduced by Zadeh \cite{zadeh_fuzzy} in 1965. In 1967, Goguen \cite{goguen_sets} extended this concept to $L$-fuzzy sets. In this section, we recall the basic definitions and results associated with $L$-subsets that shall be used throughout the rest of this work. These definitions can be found in chapter 1 of \cite{mordeson_comm}.

Let $X$ be a non-empty set. An $L$-subset of $X$ is a function from $X$  into $L$. The set of  $L$-subsets of $X$ is called  the $L$-power set of $X$ and is denoted by $L^X$.  For  $\mu \in L^X, $  the set $ \lbrace\mu(x) \mid x \in X \rbrace$  is called the image of $\mu$  and is denoted by  Im $\mu $. The tip and tail of $ \mu $  are defined as $\bigvee \limits_{x \in X}\mu(x)$ and $\bigwedge \limits_{x \in X}\mu(x)$, respectively. An $L$-subset $\mu$ of $X$ is said to be contained in an $L$-subset $\eta$  of $X$ if  $\mu(x)\leq \eta (x)$ for all $x \in X$. This is denoted by $\mu \subseteq \eta $.  For a family $\lbrace\mu_{i} \mid i \in I \rbrace$  of $L$-subsets in  $X$, where $I$  is a non-empty index set, the union $\bigcup\limits_{i \in I} \mu_{i} $    and the intersection  $\bigcap\limits_{i \in I} \mu_{i} $ of  $\lbrace\mu_{i} \mid i \in I \rbrace$ are, respectively, defined by
\begin{center}
	$\bigcup\limits_{i \in I} \mu_{i}(x)= \bigvee\limits_{i \in I} \mu(x)  $ \quad and \quad $\bigcap\limits_{i \in I} \mu_{i} (x)= \bigwedge\limits_{i \in I} \mu(x) $
\end{center}
for each  $x \in X $. If  $\mu \in L^X $  and  $a \in L $,  then the level  subset $\mu_{a}$ and the strong level subset $\mu_{a}^{>}$ of $\mu$  are defined as
\begin{center}
	$\mu_{a}= \lbrace x \in X \mid \mu (x) \geq a\rbrace$
\end{center}
and 
\[ \mu_{a}^{>} = \{ x \in X \mid \mu(x) > a \} \]
respectively. For $\mu, \nu \in L^{X} $, it can be verified easily that if $\mu\subseteq \nu$, then $\mu_{a} \subseteq \nu_{a} $ for each $a\in L $.

For $a\in L$ and $x \in X$, we define $a_{x} \in L^{X} $ as follows: for all $y \in X$,
\[
a_{x} ( y ) =
\begin{cases}
a &\text{if} \ y = x,\\
0 &\text{if} \ y\ne x.
\end{cases}
\]
$a_{x} $ is referred to as an $L$-point or $L$-singleton. We say that $a_{x} $ is an $L$-point of $\mu$ if and only if
$\mu( x )\ge a$ and we write $a_{x} \in \mu$. 

Let $S$ be a groupoid. The set product $\mu \circ \eta$   of $\mu, \eta \in L^S$ is an $L$-subset of $S$ defined by
\begin{center}
	$\mu \circ \eta (x) = \bigvee \limits_{x=yz}\lbrace\mu (y) \wedge \eta (z) \rbrace.$
\end{center}

\noindent Note that if $x$ cannot be factored as  $x=yz$  in $S$, then  $\mu \circ \eta (x)$, being  the least upper bound of the empty set, is zero. It can be verified that the set product is associative in  $L^S$  if $S$ is a semigroup.

Let $f$ be a mapping from a set $X$ to a set $Y$. If $\mu \in L ^{X}$ and $\nu \in L^{Y}$, then the image $f(\mu )$
of $\mu $ under $f$ and the preimage $f^{-1} (\nu )$ of $\nu $ under $f$ are $L$-subsets of $Y$ and $X$ respectively, defined by
\[
f(\mu )(y)=\bigvee\limits_{x\in f^{-1} (y)} \{\mu (x)\} \quad \text{ and } \quad f^{-1} (\nu )(x)=\nu (f(x)).
\]
Again,  if $f^{-1} (y)=\phi $,
then $f(\mu )(y)$ being the least upper bound of the empty set, is zero.

\begin{proposition}(\cite{mordeson_comm}, Theorem 1.1.12)
	\label{hom}
	Let $f : X \rightarrow Y$ be a mapping.
	\begin{enumerate}
		\item[({i})] Let $\{ \mu_i \}_{i \in I}$ be a family of $L$-subsets of $X$. Then, $f(\mathop{\cup}\limits_{i \in I} \mu_i) = \mathop{\cup}\limits_{i \in I} f(\mu_i)$ and $f(\mathop{\cap}\limits_{i \in I} \mu_i) \subseteq \mathop{\cap}\limits_{i \in I}f(\mu_i)$.
		\item[({ii})] Let $\mu \in L^X$. Then, $f^{-1}(f(\mu)) \supseteq \mu$. The equality holds if $f$ is injective.
		\item[{(iii)}] Let  $\nu \in L^Y$. Then, $f(f^{-1}(\nu)) \subseteq \nu$. The equality holds if $f$ is surjective.
		\item[{(iv)}] Let $\mu \in L^X$ and $\nu \in L^Y$. Then, $f(\mu) \subseteq \nu$ if and only if $\mu \subseteq f^{-1}(\nu)$. Moreover, if $f$ is injective, then $f^{-1}(\nu) \subseteq \mu$ if and only if $\nu \subseteq f(\mu)$.
	\end{enumerate}
\end{proposition}

Throughout this paper, $G$ denotes an ordinary group with the identity element `$e$' and $I$ denotes a non-empty indexing set. Also, $1_A$ denotes the characteristic function of a non-empty set $A$.

In 1971, Rosenfeld \cite{rosenfeld_fuzzy} applied the notion of fuzzy sets to groups to introduce the fuzzy subgroup of a group. Liu \cite{liu_op}, in 1981, extended the notion of fuzzy subgroups to $L$-fuzzy subgroups ($L$-subgroups), which we define below.   

\begin{definition}(\cite{rosenfeld_fuzzy})
	Let $\mu \in L ^G $. Then, $\mu $ is called an $L$-subgroup of $G$ if for each $x, y\in G$,
	\begin{enumerate}
		\item[({i})] $\mu (xy)\ge \mu (x)\wedge \mu (y)$,
		\item[({ii})] $\mu (x^{-1} )=\mu (x)$.
	\end{enumerate}
	The set of $L$-subgroups of $G$ is denoted by $L(G)$. Clearly, the tip of an $L$-subgroup
	is attained at the identity element of $G$.
\end{definition}

\begin{theorem}(\cite{ajmal_homo})
	\label{lev_gp}
	Let $\mu \in L ^G $. Then, $\mu $ is an $L$-subgroup of $G$ if and only if each non-empty level subset $\mu_{a} $ is a subgroup of $G$.
\end{theorem}

\begin{theorem}(\cite{ajmal_homo})
	\label{hom_gp}
	Let $f : G \rightarrow H$ be a group homomorphism. Let $\mu \in L(G)$ and $\nu \in L(H)$. Then, $f(\mu) \in L(H)$ and $f^{-1}(\nu) \in L(G)$.
\end{theorem}

It is well known in literature that the intersection of an arbitrary family of $L$-subgroups of a group is an $L$-subgroup of the given group.

\begin{definition}(\cite{rosenfeld_fuzzy})
	Let $\mu \in L ^G $. Then, the $L$-subgroup of $G$ generated by $\mu $ is defined as the smallest $L$-subgroup of $G$
	which contains $\mu $. It is denoted by $\langle \mu \rangle $, that is,
	\[
	\langle \mu \rangle =\cap\{\mu _{{i}} \in L(G) \mid \mu \subseteq \mu _{i}\}.
	\]
\end{definition}

The concept of normal fuzzy subgroup of a group was introduced by Liu \cite{liu_inv} in 1982. We define the normal $L$-subgroup of a group $G$ below:  

\begin{definition}(\cite{liu_inv})
	Let $\mu\in L(G)$. Then, $\mu $ is called a normal $L$-subgroup of $G$ if for all $x, y \in  G$, $\mu ( xy ) = \mu ( yx )$.
\end{definition}

\noindent The set of normal $L$-subgroups of $G$ is denoted by $NL(G)$. 

\begin{theorem}(\cite{mordeson_comm}, Theorem 1.3.3)
	\label{lev_norgp}
	Let $\mu \in L{(G)}$. Then, $\mu \in NL(G)$ if and only if each non-empty level subset $\mu_a$ is a normal subgroup of $G$.	
\end{theorem}

Let $\eta, \mu\in L^{{G}}$ such that $\eta\subseteq\mu$. Then, $\eta$ is said to be an $L$-subset of $\mu$. The set of all $L$-subsets of $\mu$ is denoted by $L^{\mu}.$
Moreover, if $\eta,\mu\in L(G)$ such that  $\eta\subseteq \mu$, then $\eta$ is said to be an $L$-subgroup of $\mu$. The set of all $L$-subgroups of $\mu$ is denoted by $L(\mu)$.

From now onwards, $\mu$ denotes an $L$-subgroup of $G$ which shall be considered as the parent $L$-group. In fact, $\mu$ is an $L$-subgroup of $G$ if and only if $\mu$ is an $L$-subgroup of $1_G$.

\begin{definition}(\cite{ajmal_sol}) 
	Let $\eta\in L(\mu)$ such that $\eta$ is non-constant and $\eta\ne\mu$. Then, $\eta$ is said to be a proper $L$-subgroup of $\mu$.
\end{definition}

\noindent Clearly, $\eta$ is a proper $L$-subgroup of $\mu$ if and only if $\eta$ has distinct tip and tail and $\eta\ne\mu$.

\begin{definition}(\cite{ajmal_nil})
	Let $\eta \in L(\mu)$. Let $a_0$ and $t_0$ denote the tip and tail of $\eta$, respectively. We define the trivial $L$-subgroup of $\eta$ as follows:
	\[ \eta_{t_0}^{a_0}(x) = \begin{cases}
	a_0 & \text{if } x=e,\\
	t_0 & \text{if } x \neq e.
	\end{cases} \]
\end{definition}

\begin{theorem}(\cite{ajmal_nil}, Theorem 2.1)
	\label{lev_sgp}
	Let $\eta \in L^\mu$. Then, 
	\begin{enumerate}
		\item[({i})] $\eta\in L(\mu)$ if and only if each non-empty level subset $\eta_a$  is a subgroup of $\mu_a$.
		\item[({ii})] $\eta \in L(\mu)$ if and only if each non-empty strong level subset $\eta_{a}^{>}$ is a subgroup of $\mu_{a}^{>}$, provided $L$ is a chain. 
	\end{enumerate}	
\end{theorem}

The normal fuzzy subgroup of a fuzzy group was introduced by Wu \cite{wu_normal} in 1981. We note that for the development of this concept, Wu \cite{wu_normal} preferred the $L$-setting. Below, we recall the notion of a normal $L$-subgroup of an $L$-group:

\begin{definition}(\cite{wu_normal})
	Let $\eta \in L(\mu)$. Then, we say that  $\eta$  is a normal $L$-subgroup of $\mu$   if  
	\begin{center}
		$\eta(yxy^{-1}) \geq \eta(x)\wedge \mu(y)$ for  all  $x,y \in G.$
	\end{center}
\end{definition}

\noindent The set of normal $L$-subgroups of $\mu$  is denoted by $NL(\mu)$. If $\eta \in NL(\mu)$, then we write\vspace{.2cm} $ \eta \triangleleft \mu$. 

Here, we mention that the arbitrary intersection of a family of normal $L$-subgroups of an $L$-group $\mu$ is again a normal $L$-subgroup of $\mu$. 

\begin{theorem}(\cite{ajmal_char})
	\label{lev_norsgp}
	Let $\eta \in L(\mu)$. Then, $\eta\in NL(\mu)$ if and only if each non-empty level subset $\eta_a$  is a normal subgroup of $\mu_a$.
\end{theorem}

\begin{definition}(\cite{rosenfeld_fuzzy})
	Let $\mu \in L^X$. Then, $\mu$ is said to possess sup-propery if for each $A \subseteq X$, there exists $a_0 \in A$ such that $ \mathop \vee \limits_{a \in A}  {\mu(a) } = \mu(a_0)$. 
\end{definition}

Lastly, recall the following form \cite{ajmal_gen, ajmal_sol}:

\begin{theorem}(\cite{ajmal_gen}, Theorem 3.1)
	\label{gen}
	Let $\eta\in L^{^{\mu}}.$ Let $a_{0}=\mathop {\vee}\limits_{x\in G}{\left\{\eta\left(x\right)\right\}}$ and define an $L$-subset $\hat{\eta}$ of $G$ by
	\begin{center}
		$\hat{\eta}\left(x\right)=\mathop{\vee}\limits_{a \leq a_{0}}{\left\{a \mid x\in\left\langle \eta_{a}\right\rangle\right\}}$.
	\end{center}
	
	\noindent Then, $\hat{\eta}\in L(\mu)$ and  $\hat{\eta} =\left\langle \eta \right\rangle$.
\end{theorem}

\begin{theorem}(\cite{ajmal_gen}, Theorem 3.7)
	\label{gen_sup}
	Let $\eta \in L^{\mu}$ and possesses the sup-property. If $a_0 = \mathop{\vee}\limits_{x \in G}\{\eta(x)\}$, then for all $b \leq a_0$, $\langle \eta_b \rangle = \langle \eta \rangle_b$.
\end{theorem}

\begin{theorem}(\cite{ajmal_sol}, Lemma 3.27)
	\label{gen_hom}
	Let $f : G \rightarrow H$ be a group homomorphism, let $\mu \in L(G)$ and $\nu \in L(H)$. Then, for all $\eta \in L^{\mu}$, $\langle f(\eta) \rangle = f(\langle \eta \rangle)$ and for all $\theta \in L^{\nu}$, $\langle f^{-1}(\theta) \rangle = f^{-1}(\langle \theta \rangle).$
\end{theorem}

\section{Maximal $L$-Subgroups of a Nilpotent $L$-Group}

In this section, we show that if $\mu$ is a nilpotent $L$-group, then the maximal $L$-subgroups of $\mu$ having the same tip and tail as $\mu$ are normal in $\mu$ (Theorem \ref{nil_max}). To prove this result, the concepts of nilpotent $L$-subgroups, normalizer of an $L$-subgroup of an $L$-group and normal closure of an $L$-subgroup in an $L$-group are used. We begin by recalling the notion of cosets and normalizer of an $L$-subgroup from \cite{ajmal_nor}: 

\begin{definition}(\cite{ajmal_nor})
	Let $\eta \in L(\mu)$ and let $a_x$ be an $L$-point of $\mu$. The left (respectively, right) coset of $\eta$ in $\mu$ with respect to $a_x$ is defined as the set product $a_x \circ \eta$ ($\eta \circ a_x$).
\end{definition}

\noindent From the definition of set product of two $L$-subsets, it can be easily seen that for all $z \in G$, 
\[ (a_x \circ \eta)(z) = a \wedge \eta(x^{-1}z) \quad \text{ and } \quad (\eta \circ a_x)(z) = a \wedge \eta(zx^{-1}). \]

\begin{definition}(\cite{ajmal_nor})
	Let $\eta \in L(\mu)$. The normalizer of $\eta$ in $\mu$, denoted by $N(\eta$), is the $L$-subgroup defined as follows:
	\[ N(\eta) = \bigcup \left\{ a_x \in \mu \mid a_x \circ \eta = \eta \circ a_x \right\}. \]
	$N(\eta)$ is the largest $L$-subgroup of $\mu$ such that $\eta$ is a normal $L$-subgroup of $N(\eta)$. Also, it has been established in \cite {ajmal_nor} that  $\eta \in NL(\mu)$ if and only if $N(\mu)=\mu$.
\end{definition}
 
 \noindent Below, we provide an $L$-point characterization for an $L$-subgroup to be contained in the normalizer.

\begin{lemma}
	\label{lpt_norm}
	Let $\eta$, $\theta \in L(\mu)$. If $a_x \circ b_y \circ a_{x^{-1}} \in \eta$ for all $a_x \in \theta$ and $b_y \in \eta$, then $\theta \subseteq N(\eta)$.
\end{lemma}
\begin{proof}
	Suppose that $a_x \circ b_y \circ a_{x^{-1}} \in \eta$ for all $a_x \in \theta$ and $b_y \in \eta$. Firstly, we show that for all $x,~y \in G$
	  \begin{equation}\eta(xyx^{-1}) \geq \eta(y) \wedge \theta(x).\end{equation}
	
	\noindent Let $x,~y \in G$ and let $a=\theta(x)$ and $b=\eta(y)$. Then, $a_x \in \theta$ and $b_y \in \eta$. Hence, by hypothesis,
	\[ (a \wedge b)_{xyx^{-1}} = a_x \circ \ b_y \circ a_{x^{-1}} \in \eta, \]
	which implies that
	\[ \eta(xyx^{-1}) \geq a \wedge b = \eta(y) \wedge \theta (x).	\]
	This proves (1). Next, in order to show that $\theta \subseteq N(\eta)$, we claim that for all $a_x \in \theta$, 
	\begin{equation} a_x \circ \eta =  \eta \circ a_x.\end{equation}
	So, we let $z \in G$ and $a_x \in \theta$. Then,
	\begin{equation*}
	\begin{split}
	(a_x \circ \eta)(z) &= a \wedge \eta(x^{-1}z) \\
	& = a \wedge \eta(x^{-1}(zx^{-1})x) \\
	& \geq a \wedge \theta(x) \wedge \eta(zx^{-1})\qquad	(\text{by~} (1) )\\
	& = a \wedge \eta(zx^{-1}) \qquad\qquad~~~ (\text{since $a_x \in \theta$}) \\
	& = (\eta \circ a_x)(z).
	\end{split}	
	\end{equation*}
	Similarly, it can be shown that $(\eta \circ a_x)(z) \geq (a_x \circ \eta)(z)$. This proves the claim.
	 Here, we write 
	$ {\ A} =\{a_x \in \mu : \eta \circ a_x = a_x \circ \eta\}$. By (2), it follows that 
	\[a_x \in  {\ A} \text {~for all~} a_x \in \theta.\]
	
\noindent Finally, let $x \in G$ 
	 and observe
	\[	\theta(x) = \bigcup_{a_x \in \theta} \{a_x\} \subseteq \bigcup _{ a_x \in {\ A}} \{a_x \} = N(\eta)(x). \]	
This completes the proof.
\end{proof}

\noindent The notion of a nilpotent $L$-subgroup was developed by Ajmal and Jahan \cite{ajmal_nil}. For this, the definition of the commutator of two $L$-subgroups was modified, and this modified definition was used to develop the notion of the descending central chain of an $L$-subgroup. We recall these concepts below.

\begin{definition}(\cite{ajmal_nil})
	Let $\eta$, $\theta \in L^{\mu}$. The commutator of $\eta$ and $\theta$ is the $L$-subset $(\eta, \theta)$ of $\mu$ defined as follows:
	\[ (\eta, \theta)(x) = 
	\begin{cases}
	\vee \{ \eta(y) \wedge \theta(z)\} & \text{if }x=[y,z] \text{ for some } y,z \in G, \\ 
	\text{inf } \eta \wedge \text{inf } \theta & \text{if } x \neq [y,z] \text{ for any } y,z \in G.
	\end{cases} \]
\end{definition}

\noindent The commutator $L$-subgroup of $\eta$, $\theta \in L^{\mu}$, denoted by $[\eta, \theta]$, is defined to be the $L$-subgroup of $\mu$ generated by $(\eta, \theta)$.

\begin{definition}(\cite{ajmal_nil})
	\label{def_centchain}
	Let $\eta \in L(\mu)$. Take $Z_0(\eta) = \eta$ and for each $i \geq 0$, define $Z_{i+1}(\eta) = [Z_i(\eta), \eta]$. Then, the chain
	\[ \eta = Z_0(\eta) \supseteq Z_1(\eta) \supseteq \ldots \supseteq Z_i(\eta) \supseteq \ldots \]
	of $L$-subgroups of $\mu$ is called the descending central chain of $\eta$. 
\end{definition}

\begin{definition}(\cite{ajmal_nil})
	\label{def_nil}
	Let $\eta \in L(\mu)$ with tip $a_0$ and tail $t_0$ and $a_0 \neq t_0$. If the descending central chain
	\[ \eta = Z_0(\eta) \supseteq Z_1(\eta) \supseteq \ldots \supseteq Z_i(\eta) \supseteq \ldots \]
	terminates to the trivial $L$-subgroup $\eta_{t_0}^{a_0}$ in a finite number of steps, then $\eta$ is called a nilpotent $L$-subgroup of $\mu$. Moreover, $\eta$ is said to be nilpotent of class $c$ if $c$ is the smallest non-negative integer such that $Z_c(\eta) = \eta_{t_0}^{a_0}$. 
\end{definition}

\noindent Here, we define the successive normalizers of $\eta$ as follows:
\[	\eta_0 = \eta \qquad \text{ and } \qquad \eta_{i+1} = N(\eta_i) \qquad \text{ for all } i \geq 0.	\]
Then, by the definition of normalizer (see Definition 3.2),  $\eta_{i+1}$ is the the largest $L$-subgroup of $\mu$ containing $\eta_i$ such that  $\eta_i \lhd \eta_{i+1}$. Consequently,
\begin{equation}
	\label{chain1}
	\eta = \eta_0 \subseteq \eta_1 \subseteq \ldots \subseteq \eta_i \subseteq \eta_{i+1} \subseteq \ldots 
\end{equation}
is an ascending chain of $L$-subgroups of $\mu$ starting from $\eta$ such that each $\eta_i$ is a normal $L$-subgroup of $\eta_{i+1}$. We call (3) the ascending chain of normalizers of  $\eta$ in $\mu$.
\begin{lemma}
	\label{subnormal}
	Let $\mu \in L(G)$  be a nilpotent $L$-group and $\eta$ be an $L$-subgroup of $\mu$ having the same tip and tail as $\mu$. Then, the ascending chain of normalizers of  $\eta$ in $\mu$ is finite and terminates at $\mu$.
\end{lemma}
\begin{proof}
	Let $\mu$ be a nilpotent $L$-group and let $\eta$ be an $L$-subgroup of $\mu$ having the same tip and tail as $\mu$. Let $a_0$ and $t_0$ denote the tip and tail of $\mu$, respectively. Let
	\[	\mu = Z_{0}(\mu) \supseteq Z_{1}(\mu) \supseteq \ldots \supseteq Z_{c} (\mu) = \mu^{a_0}_{t_0}	\]
	be the descending central chain of $\mu$, where
	\[	Z_i(\mu) = [Z_{i-1}(\mu), \mu] \qquad \text{ for all } i \geq 1.	\]
	Define the successive normalizers of $\eta$ as follows:
	\[	\eta_0 = \eta \qquad \text{ and } \qquad \eta_{i+1} = N(\eta_i) \qquad \text{ for all } i \geq 0.	\]
	Then,
	\begin{equation}
	\label{chain1}
	\eta = \eta_0 \subseteq \eta_1 \subseteq \ldots \subseteq \eta_i \subseteq \eta_{i+1} \subseteq \ldots 
	\end{equation}
	is the ascending chain of normalizers of  $\eta$ in $\mu$. We claim that \[Z_{c-i} (\mu)\subseteq \eta_i \text{ ~~for all}~~ i \geq 0.\] 
	
	\noindent We apply induction on $i$. When $i=0$,
	\[	Z_c (\mu) = \mu^{a_0}_{t_0} = \eta^{a_0}_{t_0} \subseteq \eta = \eta_0.	\]
	Hence, the claim is true for $i=0$. Now, suppose that the claim holds for some $i \geq 0$, that is, $Z_{c-i} (\mu) \subseteq \eta_i$. We shall show that 
	\[ Z_{c-i-1} (\mu) \subseteq \eta_{i+1}. \]
	
	\noindent So, let $a_x \in Z_{c-i-1} (\mu)$. Then, for all $b_y \in \eta_i$, 
	\begin{equation*}
	\begin{split}
	a_x \circ b_y \circ a_{x^{-1}} &= (a \wedge b)_{xyx^{-1}} \\
	& = (a \wedge b)_{[x,y]y} \\
	& = (a \wedge b)_{[x,y]} \circ b_y \\
	& \in [Z_{c-i-1} (\mu), \mu] \circ \eta_i \qquad (\text{since } a_x \in Z_{c-i-1} (\mu), b_y \in \mu)\\
	& = Z_{c-i} (\mu) \circ \eta_i , ~~~~~~~~~~~~(\text{by definition,} ~ [Z_{c-i-1} (\mu), \mu] = Z_{c-i}(\mu))\\
	& \subseteq \eta_i. ~~~~~~~~~~~~~~~~~~~~~~~~~~ (\text{by the induction hypothesis}, Z_{c-i} \subseteq \eta_i )
	\end{split}	
	\end{equation*}
	Therefore, by Lemma \ref{lpt_norm},
	\[ Z_{c-i-1} (\mu) \subseteq N(\eta_i) = \eta_{i+1}.	\]
	Hence, the claim holds.
	In particular, when $i=c$, $\mu=Z_0 \subseteq \eta_c$. Therefore, $\eta_c = \mu$. This proves the result.	
\end{proof}

\noindent We now recall the definition of the normal closure of an $L$-subgroup in an $L$-group from \cite{ajmal_nc}:

\begin{definition}(\cite{ajmal_nc})
	Let $\eta \in L(\mu)$. The $L$-subset $\mu\eta\mu^{-1}$ of $\mu$ defined by 
	\[ \mu\eta\mu^{-1}(x) = \bigvee_{x=zyz^{-1}} \left\{ \eta(y) \wedge \mu(z) \right\} \qquad \text{ for each } x \in G \]
	is called the conjugate of $\eta$ in $\mu$.
\end{definition}

\begin{definition}(\cite{ajmal_nc})
	Let $\eta \in L(\mu)$. The normal closure of $\eta$ in $\mu$, denoted by $\eta^\mu$, is defined to be the $L$-subgroup of $\mu$ generated by the conjugate $\mu\eta\mu^{-1}$, that is,
	\[ \eta^\mu = \langle \mu\eta\mu^{-1} \rangle. \]
	Moreover, $\eta^\mu$ is the smallest normal $L$-subgroup of $\mu$ containing $\eta$.
\end{definition}

\begin{definition}(\cite{ajmal_nc})
	Let $\eta \in L(\mu)$. Define a descending series of $L$-subgroups of $\mu$ inductively as follows:
	\[ \eta^{(0)} = \mu \qquad \text{ and } \qquad \eta^{(i)} = \eta^{\eta^{(i-1)}} \qquad \text{ for all } i \geq 1. \]
	Then, $\eta^{(i)}$ is the smallest normal $L$-subgroup of $\eta^{(i-1)}$ containing $\eta$. We call $\eta^{(i)}$ the $i^{th}$ normal closure of $\eta$ in $\mu$. The series of $L$-subgroups
	\[ \mu = \eta^{(0)} \supseteq \eta^{(1)} \supseteq \ldots \supseteq \eta^{(i-1)} \supseteq \eta^{(i)} \supseteq \dots \supseteq \eta \]
	is called the normal closure series of $\eta$ in $\mu$.
\end{definition}

\begin{theorem} (\cite{ajmal_nc}) \label{nor_nc} Let $\eta \in L(\mu)$. Then, $\eta \in NL(\mu)$ if and only if $\eta ^{\mu} = \eta$.\end{theorem}
\begin{proposition}
	\label{int_nor}
	Let $\eta \in NL(\mu)$ and $\theta \in L(\mu)$. Then, $\eta \cap \theta \in NL(\theta)$.
\end{proposition}

\begin{lemma}
	\label{chain_nc}
	Let $\mu \in L(G)$ and $\eta \in L(\mu)$. Then, there exists an ascending chain of $L$-subgroups 
	\[ \eta = \theta_0 \subseteq \theta_1 \subseteq \ldots \subseteq \theta_n \subseteq \ldots \subseteq  \mu \]
	terminating at $\mu $ in a  finite number of steps such that each $\theta_i$ is normal in $\theta_{i+1}$ if and only if the normal closure series of $\eta$ in $\mu$ terminates at $\eta$ in a  finite number of steps.
\end{lemma}
\begin{proof}
	Suppose 
	\[ \eta = \theta_0 \subseteq \theta_1 \subseteq \ldots \subseteq \theta_n = \mu \]
is a finite ascending chain of $L$-subgroups such that each $\theta_i$ is normal in $\theta_{i+1}$. Let $\eta^{(i)}$ denote the $i^{th}$ normal closure of $\eta$ in $\mu$. We claim that \[\eta^{(i)} \subseteq \theta_{n-i}~~\text{ for all}~~ i \geq 0.\]
	
	\noindent We apply induction on $i$. When $i=0$, \[ \eta^{(0)} = \mu = \theta_n. \]
	Hence, the claim is true for $i=0$. Now, suppose that the claim holds for some $i \geq 0$, that is, $\eta^{(i)} \subseteq \theta_{n-i}$. We shall show that
	\[\eta^{(i +1)} \subseteq \theta_{n-i-1}. \]
	Here, note that by the hypothesis, $\theta_{n-i-1}$ is normal in $\theta_{n-i}$. Also, by the induction hypothesis, $\eta^{(i)} \subseteq \theta_{n-1}$. Therefore, by Proposition \ref{int_nor}, we have  
	\[\theta_{n-i-1} \cap \eta^{(i)} \lhd \eta^{(i)}.\] 
	Moreover, again by the hypothesis, $\eta \subseteq \theta_{n-i-1}$. Also, by the definition of normal closure (see Definition 3.11), $\eta \subseteq \eta^{(i)}$. This implies,  
	\[ \eta \subseteq \theta_{n-i-1} \cap \eta^{(i)} \subseteq \eta^{(i)}.\] 
	Again, by the definition of normal closure, $\eta^{(i+1)}$ is the smallest normal $L$-subgroup of $\eta^{(i)}$ containing $\eta$. Therefore,
	\[ \eta^{(i+1)} \subseteq \theta_{n-i-1} \cap \eta^{(i)} \subseteq \theta_{n-i-1}. \]
	This proves the result by induction. In particular, when $i=n$, 
	\[ \eta \subseteq \eta^{(n)} \subseteq \theta_0 = \eta. \]
	Thus, $\eta^{(n)} = \eta$ and the normal closure series of $\eta$ in $\mu$ terminates  at $\eta$ in $n$ steps.
	
	\noindent Conversely, if $\eta^{(n)} = \eta$ for some non-negative integer $n$, then the series of $L$-subgroups
	\[ \eta = \eta^{(n)} \subseteq \eta^{(n-1)} \subseteq \ldots \subseteq \eta^{(0)} = \mu \]
	is a finite ascending chain such that each $\eta^{(i)}$ is normal in $\eta^{(i-1)}$. 
\end{proof}

\noindent Now, we recall the definition of a maximal $L$-subgroup of an $L$-group from \cite{jahan_max}:

\begin{definition}
	Let $\mu \in L(G)$. A proper $L$-subgroup $\eta$ of $\mu$ is said to be a maximal $L$-subgroup of $\mu$ if, whenever $\eta \subseteq \theta \subseteq \mu$ for some $\theta \in L(\mu)$, then either $\theta = \eta$ or $\theta = \mu$.
\end{definition}

\begin{theorem}
	\label{nil_max}
	Let $\mu \in L(G)$ be a nilpotent $L$-group. Then, every maximal $L$-subgroup of $\mu$ having the same tip and tail as $\mu$ is normal in $\mu$.
\end{theorem}
\begin{proof}
	Let $\eta$ be a maximal $L$-subgroup of $\mu$ having the same tip and tail as $\mu$. Since $\mu$ is nilpotent,
	by Lemma \ref{subnormal}, the ascending chain of normalizers of $\eta$ in $\mu$ is finite and terminates  at $\mu$ after finitely many steps.   Therefore, by Lemma \ref{chain_nc}, the normal closure series of $\eta$ in $\mu$ is also finite and terminates at $\eta$ in finitely many steps, say $m$. Let the finite normal closure series be 
	\begin{equation}	\eta = \eta^{(m)} \subseteq \eta^{(m-1)} \subseteq \ldots \subseteq \eta^{(0)} = \mu,	\end{equation}
	where $\eta^{(i)}=\eta^{\eta^{(i-1)}}$.
	\noindent Here, we note that $\eta^{(1)} \neq \eta^{(0)}$, for if $\eta^{(1)} = \eta^{(0)}$, then 
	\[ \eta^{(2)} = \eta^{\eta^{(1)}} = \eta^{\eta^{(0)}} = \eta^{(1)} = \eta^{(0)}, \]
	which implies that $\eta=\eta^{(m)} = \eta^{(0)}=\mu$, which contradicts the fact that $\eta \subsetneq \mu$. Therefore, we have
	\begin{equation}\eta \subseteq \eta^{(1)} \subsetneq \eta^{(0)} = \mu.\end{equation}
	Since $\eta$ is a maximal $L$-subgroup of $\mu$,  we must have $ \eta = \eta^{(1)} =\eta^{\mu}$. By Theorem \ref{nor_nc}, it follows that $\eta$ is a normal $L$-subgroup of $\mu$.
\end{proof}

\noindent Below, we provide an example of a nilpotent $L$-group and its maximal normal $L$-subgroup $\eta$ :

\begin{example}
	Let $M=\{ l,f,a,b,c,d,u \}$ be the lattice given by figure 1. Let $\mathbf{2}$ be the chain $0<1$. Then, 
	\begin{align*}
	M \times \mathbf{2} = &\{ (l,0), (f,0), (a,0), (b,0), (c,0), (d,0), (u,0),\\
	&~ (l,1), (f,1), (a,1), (b,1), (c,1), (d,1), (u,1) \}. 
	\end{align*} 
	\noindent Let $G=S_4$, the group of all permutations of the set $\{1,2,3,4\}$ with identity element $e$. Let
	\[ D_4^1 = \langle (24), (1234) \rangle, D_4^2 = \langle (12), (1324) \rangle, D_4^3 = \langle (23), (1342) \rangle \]
	denote the dihedral subgroups of $G$ and $V_4 = \{e, (12)(34), (13)(24), (14)(23)\}$ denote the Klein-4 subgroup of $G$. 
	\begin{center}
		\includegraphics[scale=1.0]{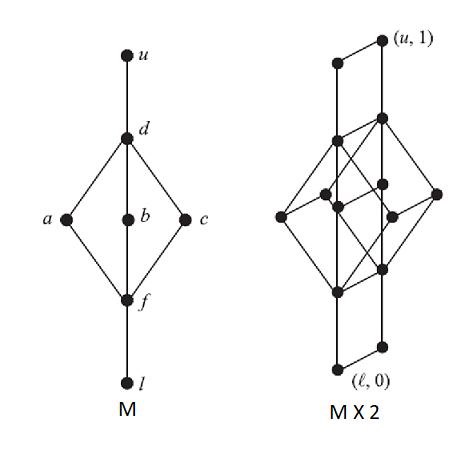}\\
		\centering \text{Figure 1}		
	\end{center}
	Define the $L$-subset $\mu$ of $G$ as follows:
	\[ \mu(x) = \begin{cases}
	(u,1) &\text{if } x=e,\\
	(d,1) &\text{if } x \in V_4 \setminus \{e\},\\
	(a,1) &\text{if } x \in D_4^1 \setminus V_4,\\
	(b,1) &\text{if } x \in D_4^2 \setminus V_4,\\
	(c,1) &\text{if } x \in D_4^3 \setminus V_4,\\
	(f,0) &\text{if } x \in S_4 \setminus \mathop{\cup}\limits_{i=1}^3 D_4^i.
	\end{cases} \]
	Since each non-empty level subset $\mu_t$ is a subgroup of $G$, by Theorem \ref{lev_gp}, $\mu$ is an $L$-subgroup of $G$. We show that $\mu$ is a nilpotent $L$-group. Since $G' = A_4$, it is easy to verify that the commutator $(\mu,\mu)$ is given by
	\[ (\mu,\mu)(x) = \begin{cases}
	(u,1) &\text{if } x=e,\\
	(a,1)\wedge(d,1)=(a,1) &\text{if } x=(13)(24),\\
	(b,1)\wedge(d,1)=(b,1) &\text{if } x=(12)(34),\\
	(c,1)\wedge(d,1)=(c,1) &\text{if } x=(14)(23),\\
	(f,0) &\text{if } x \in S_4 \setminus V_4.
	\end{cases} \]
	Since every non-empty level subset $(\mu,\mu)_t$ is a subgroup of $\mu_t$, by Theorem \ref{lev_sgp}, $(\mu,\mu)$ is an $L$-subgroup of $\mu$ and hence, $Z_1(\mu)=[\mu,\mu]=(\mu,\mu)$. Next, it can be easily verified that $Z_2(\mu)=[\mu,Z_1(\mu)]$ is given by
	\[ Z_2(\mu)= \begin{cases}
	(u,1) &\text{if } x=e,\\
	(f,0) &\text{if } x\in S_4\setminus\{e\}.
	\end{cases} \]
	Hence, $Z_2(\mu)=\mu_{t_0}^{a_0}$, where $a_0 = (u,0)$ and $t_0=(f,0)$ are the tip and tail of $\mu$, respectively. We conclude that $\mu$ is a nilpotent $L$-group. Now, define $\eta \in L^{\mu}$ as follows:
	\[ \eta(x) = \begin{cases}
	(u,1) &\text{if } x=e,\\
	(d,1) &\text{if } x \in V_4 \setminus \{e\},\\
	(a,1) &\text{if } x \in D_4^1 \setminus V_4,\\
	(b,0) &\text{if } x \in D_4^2 \setminus V_4,\\
	(c,1) &\text{if } x \in D_4^3 \setminus V_4,\\
	(f,0) &\text{if } x \in S_4 \setminus \mathop{\cup}\limits_{i=1}^3 D_4^i.
	\end{cases} \]
	Since every non-empty level subset $\eta_t$ is a subgroup of $\mu_t$, by Theorem \ref{lev_sgp}, $\eta \in L(\mu)$. We show that $\eta$ is a maximal $L$-subgroup of $\mu$.
	
	\noindent Suppose there exists $\theta \in L(\mu)$ such that $\eta \subsetneq \theta \subseteq \mu$. Since $\eta(x)=\mu(x)$ for all $x \in S_4 \setminus (D_4^2 \setminus V_4)$, $\theta(x) = \eta(x)$ for all $x \in S_4 \setminus (D_4^2 \setminus V_4)$. Hence, there exists $x_0 \in D_4^2 \setminus V_4$ such that 
	\[ (b,1) = \mu(x_0) \geq \theta(x_0) > \eta(x_0) = (b,0), \]
	that is, $\theta(x_0) = (b,1)$. But then $\theta(x) = (b,1) = \mu(x)$ for all $x \in D_4^2 \setminus V_4$. This implies $\theta = \mu$. We conclude that $\eta$ is a maximal $L$-subgroup of $\mu$. Finally, note that $\eta_t = \mu_t$ for all $t \in L \setminus \{(b,1)\}$ and $\eta_{(b,1)} = V_4$ is a normal subgroup of $\mu_{(b,1)} = D_4^2$. Hence, $\eta_t$ is a normal subgroup of $\mu_t$ for each $t \in L$. By Theorem \ref{lev_norsgp}, $\eta$ is a normal $L$-subgroup of $\mu$.	
\end{example}

\section{Finitely Generated $L$-subgroups}

In this section, we introduce the concepts of maximal condition on $L$-subgroups and finitely generated $L$-groups. These notions are closely related in classical group theory. A similar relation between these concepts has been provided in Theorem \ref{max_fin}. 

\begin{definition}
	Let $\mu \in L(G)$. Then, $\mu$ is said to satisfy the maximal condition  if every non-empty set $S$ of $L$-subgroups of $\mu$ contains a maximal element.
\end{definition}

\begin{theorem}
	\label{mcon_chain}
	An $L$-subgroup $\mu$ satisfies the maximal condition  if and only if each proper ascending chain of $L$-subgroups of $\mu$ is finite.
\end{theorem}
\begin{proof}
	Suppose that $\mu$ satisfies the maximal condition . Let
	\begin{equation}
	\label{chain2}
	\eta_1 \subsetneq \eta_2 \subsetneq \ldots \subsetneq \eta_n \subsetneq \ldots
	\end{equation} 
	be a proper ascending chain of $L$-subgroups in $\mu$. Let $S = \{ \eta_i \mid i \in \mathbb{N} \}$. Then, by hypothesis, $S$ contains a maximal element $\theta$, that is, there exists $\theta \in S$ such that $\theta$ is not contained in any other $\eta_i$. Since (\ref{chain2}) is a proper chain, it must terminate at $\theta$. 
	
	\noindent Conversely, suppose that each proper ascending chain of $L$-subgroups in $\mu$ is finite. Let $S$ be a non-empty set of $L$-subgroups of $\mu$. We show that $S$ contains an $L$-subgroup $\theta$ which is not contained in any other element of $S$. Since $S$ is non-empty, there exists $\eta_0$ in $S$. If $\eta_0$ is not contained in any other element of $S$, set $\theta = \eta_0$. Otherwise, there exists $\eta_1 \in S$ such that $\eta_0 \subsetneq \eta_1$. Again, if $\eta_1$ is not contained in any other element of $S$, set $\theta = \eta_1$. Otherwise, there exists $\eta_2 \in S$ such that $\eta_1 \subsetneq \eta_2$. Continuing in this manner, we get an ascending chain
	\begin{equation}
	\label{chain3}
		\eta_0 \subsetneq \eta_1 \subsetneq \eta_2 \subsetneq \ldots
	\end{equation}
	By hypothesis, (\ref{chain3}) is finite. Hence, there exists $\theta \in S$ which is not contained in any other element of $S$. Hence, the result.
\end{proof}

\begin{theorem}
	\label{mcon_subgp}
	Let $\mu \in L(G)$ satisfy the maximal condition and let $\eta \in L(\mu)$. Then, $\eta$ also satisfies the maximal condition.
\end{theorem}

\noindent In \cite{jaballah_Mingen}, Jaballah and Mordeson have introduced the notion of finitely generated fuzzy ideals of a ring through fuzzy singletons:

\begin{definition}(\cite{jaballah_Mingen})
	A fuzzy ideal $\mu$ of a ring R is said to be finitely generated if there exists a finite set of fuzzy singletons $\phi$ such that $\mu = \langle \phi \rangle$.
\end{definition}

\noindent Below, we define finitely genearted $L$-subgroup of a group in the spirit of the above definition.  

\begin{definition}
	An $L$-subgroup $\mu$ of $G$ is said to be finitely generated if there exist finitely many $L$-points ${a_1}_{x_1}, \ldots , {a_n}_{x_n}$ in $\mu$ such that $\mu = \langle {a_1}_{x_1}, \ldots , {a_n}_{x_n} \rangle$.
\end{definition}

\noindent Recall the definition of a minimal generating set for a fuzzy set as given in \cite{mordeson_comm}.
\begin{definition}(\cite{mordeson_comm})
	Let $\mu \in F(G)$. Let $S$ denote a set of fuzzy singletons such that if $x_a, x_b \in S$, then $a=b>0$ and $x \neq 0$. Define the fuzzy subset $\chi(S)$ of $G$ as follows: For all $x \in G$
	\[ \chi(S)(x) = \begin{cases}
		a \qquad \text{ if } x_a \in S,\\
		0 \qquad \text{ otherwise}.
	\end{cases} \]
	Set $\langle S \rangle = \langle \chi(S) \rangle$.
\end{definition}

\noindent Note that the following condition 
\[ \text{if } x_a, x_b \in S, \text{ then } a=b>0, \]
becomes redundant  in Definition 4.5 as the concept of finitely generated $L$-subgroup involves the union of $L$-points. Hence, Definition 4.5 is a simple version of the above definition.  

\noindent Here, we prove the following:

\begin{theorem}
	Let $H$ be a subgroup of $G$. Then, $H$ is finitely generated if and only if the characteristic function $1_H$ is a finitely generated $L$-subgroup of $G$.
\end{theorem}
\begin{proof}
	Let $H$ be a finitely generated subgroup of $G$. Then, there exist finitely many elements $x_1, x_2, \ldots, x_n$ in $H$ such that 
	\[ H = \langle x_1, x_2, \ldots, x_n \rangle. \]
	We claim that 
	\[ 1_H = \langle 1_{x_1}, 1_{x_2}, \ldots 1_{x_n} \rangle. \]
	Since $1_{x_1}, 1_{x_2}, \ldots, 1_{x_n} \in 1_H$,
	\[ \langle 1_{x_1}, 1_{x_2}, \ldots, 1_{x_n} \rangle \subseteq 1_H. \]
	To show the reverse inclusion, let $1_z \in 1_H$. Then, $z \in H = \langle x_1, x_2, \ldots, x_n \rangle.$ This implies
	\[ z = y_1 y_2 \ldots y_k, \text{ where } y_i \text{ or } {y_i}^{-1} \in \{x_1, x_2, \ldots, x_n\}. \]
	Thus,
	\begin{equation*}
		\begin{split}
			1_z &= 1_{y_1 y_2 \ldots y_k} \\
			&= 1_{y_1}\circ 1_{y_2} \circ \ldots 1_{y_k}\\
			&\subseteq \langle 1_{y_1}, 1_{y_2} \ldots, 1_{y_k} \rangle \\
			&\subseteq \langle 1_{x_1}, 1_{x_2}, \ldots, 1_{x_n} \rangle.
		\end{split}
	\end{equation*}
	Hence, $1_H = \langle 1_{x_1}, 1_{x_2}, \ldots 1_{x_n}\rangle$.
	
	\noindent Conversely, suppose that $1_H$ is finitely generated. Then, by definition, there exist finitely many $L$-points ${a_1}_{x_1}, {a_2}_{x_2}, \ldots {a_n}_{x_n}$ of $1_H$ such that
	\[ 1_H = \langle {a_1}_{x_1}, {a_2}_{x_2}, \ldots {a_n}_{x_n} \rangle. \] 
	Clearly, we can assume that each $a_i>0$. This implies
	\[ 1_H(x_i) \geq a_i > 0 \text{ for all } i=1,2,\ldots n, \]
	that is, 
	\[ 1_H(x_i) = 1. \]
	Thus, $1_{x_1}, 1_{x_2}, \ldots, 1_{x_n} \in 1_H$ and therefore,
	\[ 1_H = \langle {a_1}_{x_1}, {a_2}_{x_2}, \ldots {a_n}_{x_n} \rangle \subseteq \langle 1_{x_1}, 1_{x_2}, \ldots, 1_{x_n} \rangle \subseteq 1_H.  \]
	Thus, $1_H = \langle 1_{x_1}, 1_{x_2}, \ldots 1_{x_n} \rangle$. We claim that $H = \langle x_1, x_2, \ldots, x_n \rangle$.
	
	\noindent Suppose, if possible, that $H \neq \langle x_1, x_2, \ldots, x_n \rangle.$ Then, there exists $z \in H$ such that $z \notin \langle x_1, x_2, \ldots x_n \rangle$. Note that since $1_H = \langle \bigcup_{i=1}^{n} 1_{x_i} \rangle$, by Theorem \ref{gen}, 
	\begin{equation*}
		\begin{split}
			1 = 1_H(z) &= \mathop{\vee}\limits_{a \leq 1}{\left\{a \mid z\in\left\langle \left( \mathop{\cup}\limits_{i=1}^{n} 1_{x_i} \right)_{a}\right \rangle \right\}}\\ 
			&= \mathop{\vee}\limits_{a \leq 1}{\left\{a \mid z\in\left\langle \mathop{\cup}\limits_{i=1}^{n} (1_{x_i})_a \right \rangle \right\}}.
		\end{split}
	\end{equation*}
	Now, since $z \notin \langle x_1, x_2, \ldots x_n\rangle$, $z\notin\left\langle \mathop{\cup}\limits_{i=1}^{n} (1_{x_i})_a \right \rangle$ for any value of $a>0$. It follows that $1=1_H(z)=0$, which is absurd. Hence, we conclude that
	\[ H= \langle x_1, x_2, \ldots, x_n \rangle. \]	
\end{proof}

\noindent The following example illustrates Definition 4.5:

\begin{example}
	Let $G = D_8$, the dihedral group of order 8, given by
	\[ D_8 = \langle r, s \mid r^4 = s^2 = e, rs = sr^{-1} \rangle. \]
	Let $H_1 = \{e, r^2\}$ be the center of $G$, $H_2 = \{e, s\}$ and $K = \{ e, r^2, s, sr^2\}$.
	Let $L$ be the lattice represented by the following figure.
	\begin{center}
		\includegraphics[scale=0.7]{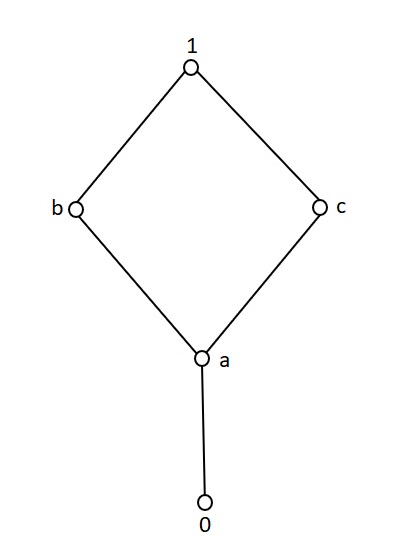}
	\end{center}
	Define $\mu : G \rightarrow L$ as follows:
	\[ \mu(x) = \begin{cases}
		1 & \text{if } x=e,\\
		b & \text{if } x \in H_1 \setminus \{e\},\\
		c & \text{if } x \in H_2 \setminus \{e\},\\
		a & \text{if } x \in K \setminus H_1 \cup H_2,\\
		0 & \text{if } x \in D_8 \setminus K.
		\end{cases} \]
	Since each non-empty level subset $\mu_t$ is a subgroup of $G$, by Theorem \ref{lev_gp},  $\mu \in L(G)$. We show that $\mu = \langle b_{r^2}, c_s \rangle$.
	\noindent Let $\eta = b_{r^2} \cup c_s$. Then, $\eta \in L^{\mu}$ and hence, $\langle \eta \rangle \subseteq \mu$. Now, tip($\eta$) = $\mathop{\vee}\limits_{x \in G} \eta(x) = b \vee c = 1$. Hence, by Theorem \ref{gen}, for all $x \in G$,
	\[ \langle \eta \rangle (x) = \mathop{\vee}\limits_{t \leq 1} \{t \mid x \in \langle \eta_t \rangle \}. \]
	Note that 
	\[ \eta_0 = D_8, \eta_a = \{r^2, s\}, \eta_b = \{r^2\}, \eta_c = \{s\}, \eta_1 = \emptyset. \]
	Hence,
	\[ \langle \eta_0 \rangle = D_8, \langle \eta_a \rangle = K, \langle \eta_b \rangle = H_1 , \langle \eta_c \rangle = H_2, \langle \eta_1 \rangle = \{e\}. \]
	Thus, $\langle \eta \rangle = \langle b_{r^2}, c_s \rangle = \mu.$ 
\end{example}

\noindent Recall that a lattice $L$ is said to be upper well ordered if  every non-empty subset of $L$ contains its supremum. Clearly, if $L$ is upper well ordered, then each $L$-subset $\eta$ satisfies the sup-property. Also, every upper well ordered lattice is a chain.

\begin{lemma}
	\label{fgen_chain}
	Let $L$ be an upper well ordered lattice and let $\mu$ be a finitely generated $L$-group. Then, every proper ascending chain of $L$-subgroups of $\mu$ whose set union equals $\mu$ is finite.
\end{lemma}
\begin{proof}
	Suppose that $\mu = \langle {a_1}_{x_1}, \ldots , {a_n}_{x_n} \rangle.$ Let
	\[ \eta_1 \subsetneq \eta_2 \subsetneq \ldots \subsetneq \eta_l \subsetneq \ldots \]
	be a proper ascending chain of $L$-subgroups of $\mu$ such that $\cup \eta_i = \mu$. Then,
	\[ {a_j}_{x_j} \in \cup \eta_i \text{ for all } j = 1, \ldots, n. \]
	This implies 
	\[ \vee \eta_i(x_j) \geq a_j \text{ for all } j = 1, \ldots, n. \]
	Since $L$ is upper well ordered, there exists $m_j$ such that $\eta_{m_j}(x_j) \geq a_j$, that is, ${a_j}_{x_j} \in \eta_{m_j}$. Let $m$ be the maximum of $m_1, \ldots, m_n.$ Then, ${a_j}_{x_j} \in \eta_m$ for all $j = 1, \ldots, n$. Thus, $\mu = \langle {a_1}_{x_1}, \ldots, {a_n}_{x_n} \rangle \subseteq \eta_m.$ Hence, the chain terminates at some $k \leq m.$
\end{proof}

\begin{lemma}  (\cite{jahan_max})
	\label{union_subgp}
	Let $L$ be a chain and let $\mu \in L(G)$. Let $C$ be a chain of $L$-subgroups of $\mu$. Then,
	\[ \bigcup_{\eta \in C} \eta \in L(\mu). \]	
\end{lemma}

\begin{theorem}
	\label{max_fin}
	Let $\mu \in L(G)$. If $\mu$ satisfies the maximal condition, then every $L$-subgroup of $\mu$ is finitely generated. The converse holds if $L$ is an upper well ordered lattice.
\end{theorem}
\begin{proof}
	Suppose that $\mu$ satisfies the maximal condition and let $\eta \in L(\mu)$. Let ${a_1}_{x_1} \in \eta$. If $\eta=\langle {a_1}_{x_1} \rangle$, then $\eta$ is finitely generated and we are done. Hence, suppose that $\eta \neq \langle {a_1}_{x_1} \rangle$. Then, there exists ${a_2}_{x_2} \in \eta$ such that ${a_2}_{x_2} \notin \langle {a_1}_{x_1} \rangle$. If $\eta = \langle {a_1}_{x_1}, {a_2}_{x_2} \rangle$, then we are done. Otherwise, there exists ${a_3}_{x_3} \in \eta$ such that ${a_3}_{x_3} \notin \langle {a_1}_{x_1}, {a_2}_{x_2} \rangle$. Continuing this process, we can define the $L$-subgroups $\eta_i$ of $\eta$ by choosing, as long as possible, $L$-points ${a_1}_{x_1}, {a_2}_{x_2}, \ldots$, of $\eta$ such that
	\[ \eta_i = \langle {a_1}_{x_1}, \ldots {a_i}_{x_i} \rangle \text{ and } {{a_{i+1}}_{x_{i+1}}} \notin \eta_i. \]
	Then, $\eta_1 \subsetneq \eta_2 \subsetneq \ldots$ is a proper ascending chain of $L$-subgroups of $\eta$. By Theorem \ref{mcon_subgp}, $\eta$ satisfies the maximal condition. Hence, by Theorem \ref{mcon_chain}, this chain terminates at some finite $m$. Thus, there does not exist any $b_y \in \eta$ such that $b_y \notin \langle {a_1}_{x_1}, \ldots {a_m}_{x_m} \rangle = \eta_m$. This implies $\eta = \eta_m$ and hence, $\eta$ is finitely generated.
	
	\noindent Conversely, suppose that $L$ is upper well ordered and let $\mu \in L(G)$ such that every $L$-subgroup of $\mu$ is finitely generated.
	Let $\eta_1 \subsetneq \eta_2 \subsetneq \ldots$ be a proper ascending chain of $L$-subgroups of $\mu$. Since $L$ is an upper well ordered lattice, by Lemma \ref{union_subgp}, $\eta = \cup \eta_i$ is an $L$-subgroup of $\mu$. By hypothesis, $\eta$ is finitely generated. Hence, by Lemma \ref{fgen_chain}, the above chain of $L$-subgroups must be finite. Hence, $\mu$ satisfies the maximal condition. 
\end{proof}

\section{Frattini $L$-Subgroup of Finitely Generated $L$-Group}

In this section, we explore some properties of maximal and Frattini $L$-subgroups of finitely generated $L$-subgroups. The Frattini $L$-subgroup of an $L$-group has been investigated by the authors in \cite{jahan_max}. Moreover, the authors have defined the non-generators of an $L$-group and established a relation between the non-generators and the Frattini $L$-subgroup of $\mu$. We recall these below:   

\begin{definition}(\cite{jahan_max})
	Let $\mu \in L(G)$. The Frattini $L$-subgroup $\Phi(\mu)$ of $\mu$ is defined to be the intersection of all the maximal $L$-subgroups of $\mu$.
	If $\mu$ has no maximal $L$-subgroups, then we define $\Phi(\mu) = \mu$.
\end{definition}

\begin{definition}(\cite{jahan_max})
	Let $\mu \in L(G)$. An $L$-point $a_x \in \mu$ is said to be a non-generator of $\mu$ if, whenever $\langle \eta, a_x \rangle = \mu$ for some $\eta \in L^{\mu}$, then $\langle \eta \rangle = \mu$.
\end{definition}

\begin{theorem}(\cite{jahan_max})
	\label{frat}
	Let $\mu \in L(G)$ and define $\lambda \in L^\mu$ as follows:
	\[ \lambda =  \bigcup \{a_x \mid a_x \text{ is a non-generator of } \mu \}. \]
	Then, $\lambda \in L(\mu)$ and $\lambda \subseteq \Phi(\mu)$. Moreover, if $L$ is an upper well ordered lattice, then $\lambda = \Phi(\mu)$.
\end{theorem}

\begin{theorem}(\cite{jahan_max})
	\label{fra_nor}
	Let $L$ be an upper well ordered lattice and let $\mu$ be a normal $L$-subgroup of $G$. Then, $\Phi(\mu)$ is a normal $L$-subgroup of $\mu$.
\end{theorem}

\begin{theorem}
	Let $L$ be an upper well ordered lattice and $\mu \in L(G)$ be a non-trivial finitely generated $L$-group. Then, $\mu$ has a maximal $L$-subgroup.
\end{theorem}
\begin{proof}
	Since $\mu$ is finitely generated, there exist finitely many $L$-points ${a_1}_{x_1}, \ldots, {a_n}_{x_n}$ such that $\mu = \langle {a_1}_{x_1}, \ldots, {a_n}_{x_n} \rangle$. Let $S$ be the set of all proper $L$-subgroups of $\mu$. Then, $S$ is non-empty since the trivial $L$-subgroup $\mu^{a_{0}}_{t_{0}}\in S$ where $a_0=\mu(e)$ and $t_0=\text{inf } \mu$.  Let $C = \{ \eta_i \}_{i \in I}$ be a chain in $S$. We claim that
	\[ \bigcup_{i \in I} \eta_i \in S. \]
	\noindent Firstly, we note that by Lemma \ref{union_subgp}, $\mathop{\bigcup}\limits_{i \in I} \eta_i \in L(\mu)$. Hence, we only need to show that $\mathop{\bigcup}\limits_{i \in I} \eta_i$ is a proper $L$-subgroup of $\mu$. Suppose, if possible, that $\mathop{\bigcup}\limits_{i \in I} \eta_i$ is not a proper $L$-subgroup of $\mu$. That is,
	\[\mu = \langle {a_1}_{x_1}, \ldots, {a_n}_{x_n} \rangle = \mathop{\bigcup}\limits_{i \in I} \eta_i.\]
	Then, ${a_1}_{x_1}, \ldots, {a_n}_{x_n} \in \mathop{\bigcup}\limits_{i \in I} \eta_i$. This implies that
	\[ \bigvee_{i \in I} \eta_i(x_j) \geq a_j \text{ for each } 1 \leq j \leq n. \]  
	\noindent Since $L$ is upper well ordered, for each $j$, there exists an $L$-subgroup, say $\eta_{{i}_{j}}$, in $C$ such that $\eta_{{i}_{j}}(x_j) \geq a_j$.  Moreover, since $C$ is a chain, without loss of generality, we may assume that 
	\[ \eta_{{i}_{1}} \subseteq \eta_{{i}_{2}} \subseteq \ldots \subseteq \eta_{{i}_{n}}.\]
	Thus,
	\[\mu = \langle {a_1}_{x_1}, \ldots, {a_n}_{x_n} \rangle = \mathop{\bigcup}\limits_{i \in I} \eta_i = \eta_{{i}_{n}},\]
	\noindent which contradicts the assumption that the elements of $C$ are proper $L$-subgroups of $\mu$. Hence, $\mathop{\bigcup}\limits_{i \in I} \eta_i$ is a proper $L$-subgroup of $\mu$. Consequently, $\mathop{\bigcup}\limits_{i \in I} \eta_i \in S$, and we conclude that every chain in $S$ has an upper bound in $S$. By Zorn's lemma, $S$  has a maximal element, which is the required maximal $L$-subgroup of $\mu$. 
\end{proof}
\begin{theorem}
	\label{fgn_frat}
	Let $L$ be an upper well ordered lattice and let $\mu$ be an $L$-group with a finitely generated Frattini $L$-subgroup $\Phi(\mu)$. Then, the only $L$-subgroup $\eta$ of $\mu$ such that 
	\[	\eta \circ \Phi(\mu) = \mu	\]
	is $\eta = \mu$.
\end{theorem}
\begin{proof}
	Since $\Phi(\mu)$ is finitely generated, there exist a finite number of $L$-points ${a_1}_{x_1}, \ldots, {a_n}_{x_n}$ in $\Phi(\mu)$ such that 
	\[	\Phi(\mu) = \langle {a_1}_{x_1}, \ldots, {a_n}_{x_n} \rangle.	\]
	Clearly,
	\begin{equation*}
	\begin{split}
	\mu &= \eta \circ \Phi(\mu)\\ 
	& \subseteq \langle \eta, \Phi(\mu) \rangle \\
	& = \langle \eta, {a_1}_{x_1}, \ldots {a_n}_{x_n} \rangle\\
	& \subseteq \mu.
	\end{split} 
	\end{equation*}
	Now, since $L$ is upper well ordered, by Theorem \ref{frat},
	\[ \Phi(\mu) = \bigcup \{ a_x \mid a_x \text{is a non-generator of } \mu \}. \]
	Therefore, each of ${a_i}_{x_i}$ is a non-generator of $\mu$. This implies
	\begin{equation*}
	\begin{split}
	\mu & = \langle \eta, {a_1}_{x_1}, \ldots {a_n}_{x_n} \rangle\\
	&= \langle \eta, {a_1}_{x_1} \ldots {a_{n-1}}_{x_{n-1}} \rangle \\
	& \quad \quad \vdots \\
	& = \langle \eta \rangle = \eta.
	\end{split}
	\end{equation*}
	Hence, $\eta = \mu$.		
\end{proof}

\begin{lemma}
	\label{zrn}
	Let $L$ be an upper well ordered lattice and let $\mu \in L(G)$. Suppose that $\theta \in L(\mu)$ and let $a_x$ be an $L$-point of $\mu$ such that $a_x \notin \theta$. Then, there exists $\eta \in L(\mu)$ such that $\eta$ is maximal with respect to the conditions $\theta \subseteq \eta$ and $a_x \notin \eta$.
\end{lemma}
\begin{proof}
	Consider the set $S = \{ \nu \in L(\mu) \mid \theta \subseteq \nu \text{ and } a_x \notin \nu \}$. Then, $S$ is non-empty since $\theta \in S$.  Let $C = \{ \theta_i \}_{i \in I}$ be a chain in $S$. We claim that 
	\[\bigcup_{i \in I} \theta_i \in S. \]  
	By Lemma \ref{union_subgp}, $\mathop{\bigcup}\limits_{i \in I} \theta_i \in L(\mu)$. Also, $\theta \subseteq \mathop{\bigcup}\limits_{i \in I} \theta_i$. Now, since $L$ is upper well ordered and $a_x \notin \theta_i$ for all $i \in I$, $a_x \notin \mathop{\bigcup}\limits_{i \in I} \theta_i$. Hence, $\mathop{\bigcup}\limits_{i \in I} \theta_i \in S$ so that every chain in  $S$ has an upper bound in $S$. Therefore, by Zorn's lemma, $S$ has a maximal element $\eta$. This proves the result.
\end{proof}

\begin{theorem}
	\label{max_prp}
	Let $L$ be an upper well ordered lattice and let $\mu$ be a finitely generated $L$-group. Let $\eta$ be a proper $L$-subgroup of $\mu$. Then, $\eta$ is contained in a maximal $L$-subgroup of $\mu$.
\end{theorem}
\begin{proof}
	Let $\mu = \langle {a_1}_{x_1}, \ldots , {a_n}_{x_n} \rangle$. Since $\eta \subsetneq \mu $, let ${b_1}_{y_1}$ be the first of ${a_1}_{x_1}, \ldots, {a_n}_{x_n}$ such that ${b_1}_{y_1} \notin \eta$. By Lemma \ref{zrn}, there exists $\theta_1 \in L(\mu)$ such that $\theta_1$ is maximal with respect to the conditions $\eta \subseteq \theta_1$ and ${b_1}_{y_1} \notin \theta_1$. Hence, any $L$-subgroup of $\mu$ that contains $\theta_1$ also contains ${b_1}_{y_1}$. Let $\langle \theta_1, {b_1}_{y_1} \rangle = \nu_1$. If $\nu_1 = \mu$, then $\theta_1$ is  the required maximal $L$-subgroup. If not, let ${b_2}_{y_2}$ be the first of ${a_1}_{x_n}, \ldots, {a_n}_{x_n}$ not contained in $\nu_1$. Then, there exists $\theta_2 \in L(\mu)$ that is maximal with respect to the conditions that $\nu_1 \subseteq \theta_2$ and ${b_2}_{y_2} \notin \theta_2$. Since $\mu = \langle {a_1}_{x_1}, \ldots, {a_n}_{x_n} \rangle$, by continuing this process, we obtain $L$-points ${b_1}_{y_1}, {b_2}_{y_2}, \ldots, {b_i}_{y_i}$ and a finite chain of $L$-subgroups
	\[ \theta_1 \subseteq \theta_2 \subseteq \ldots \subseteq \theta_i \]
	such that $\langle \theta_i, {b_i}_{y_i} \rangle = \mu$. Thus, $\theta_i$ is the required maximal $L$-subgroup of $\mu$ containing $\eta$.  
\end{proof}

\begin{lemma}
	\label{prp_pro}
	Let $L$ be an upper well ordered lattice and let $\mu$ be a finitely generated $L$-group. Let $\eta \in NL(\mu)$. Suppose that $\eta$, $\Phi(\mu)$ and $\mu$ have the same tips. Then, $\eta \subseteq \Phi(\mu)$ if and only if there is no proper subgroup $\theta$ of $\mu$ such that $\eta \circ \theta = \mu$.
\end{lemma}
\begin{proof}
	Let $\eta \subseteq \Phi(\mu)$. Let $\theta$ be a proper $L$-subgroup of $\mu$. Then, by Theorem \ref{max_prp}, there exists a maximal $L$-subgroup $\nu$ of $\mu$ such that $\theta \subseteq \nu$. Then, $\eta \subseteq \Phi(\mu) \subseteq \nu$ and $\theta \subseteq \nu$, which implies that 
	\[ \eta \circ \theta \subseteq \nu \subsetneq \mu. \]
	Hence, there does not exist any proper $L$-subgroup $\theta$ of $\mu$ such that $\eta \circ \theta = \mu$.
	Conversely, suppose that there does not exist any proper $L$-subgroup $\theta$ of $\mu$ such that $\eta \circ \theta = \mu$. If $\eta \nsubseteq \Phi(\mu)$, then there exists a maximal $L$-subgroup $\nu$ of $\mu$ such that $\eta \nsubseteq \nu$. Since tip of $\Phi(\mu)$ equals tip of $\mu$, $\text{tip}~\nu = \text{tip}~\mu$. Therefore,
	\[ \nu \subsetneq \eta \circ \nu \subseteq \mu. \]
	Since $\nu$ is a maximal $L$-subgroup of $\mu$, we must have $\eta \circ \nu = \mu$, which contradicts our assumption. Hence, $\eta \subseteq \Phi(\mu)$.  
\end{proof}

\begin{lemma}
	\label{int_pro}
	Let $\eta$, $\theta$, $\sigma \in L(\mu)$ such that $\sigma \subseteq \eta$. Then,
	\[	\eta \cap (\theta \circ \sigma) = (\eta \cap \theta) \circ \sigma.	\]
\end{lemma}
\begin{proof}
	Let $x \in G$. Then,
	\begin{equation*}
	\begin{split}
	(\eta \cap (\theta \circ \sigma))(x) & = \eta(x) \wedge \left\{	\bigvee_{y \in G} \{ \theta(xy^{-1}) \wedge \sigma(y) \right\} \\
	& = \bigvee_{y \in G} \left\{ \eta(x) \wedge \theta(xy^{-1}) \wedge \eta(y^{-1}) \wedge \sigma(y) \right\} \\
	& \leq \bigvee_{y \in G} \{ \eta(xy^{-1}) \wedge \theta(xy^{-1}) \wedge \sigma(y) \} \\
	& = \bigvee_{y \in G} \left\{ (\eta \cap \theta)(xy^{-1}) \wedge \sigma(y) \right\} \\
	& = ((\eta \cap \theta) \circ \sigma)(x).
	\end{split}
	\end{equation*} 
	Hence, $\eta \cap (\theta \circ \sigma) \subseteq \eta \cap \theta) \circ \sigma$.
	Conversely, for all $y$, $z \in G$ such that $x=yz$,
	\begin{equation*}
	\begin{split}
	(\eta \cap \theta)(y) \wedge \sigma(z) & = \eta(y) \wedge \theta(y) \wedge \sigma(z) \\
	& = \eta(y) \wedge \theta(y) \wedge \eta(z) \wedge \sigma(z) \\
	& \leq \eta(x) \wedge (\theta(y) \wedge \sigma(z)). \\
	\end{split}
	\end{equation*}
	Taking supremum over all such $y$ and $z$, we get
	\begin{equation*}
	\begin{split}
	((\eta \cap \theta) \circ \sigma))(x) & = \bigvee_{x=yz} \{ (\eta \cap \theta)(y) \wedge \sigma(z) \} \\
	& \leq \bigvee_{x=yz} \{ \eta(x) \wedge \{ \theta(y) \wedge \sigma(z) 	\} \\
	& = \eta(x) \wedge (\bigvee_{x=yz} \{ \theta(y) \wedge \sigma(z) \} ) \\
	& = (\eta \cap (\theta \circ \sigma))(x).
	\end{split}
	\end{equation*}
	Hence,
	\[	\eta \cap (\theta \circ \sigma) = (\eta \cap \theta) \circ \sigma.	\]  	 
\end{proof}

\begin{theorem}
	Let $L$ be an upper well ordered lattice and $\mu$ be a finitely generated $L$-group. Let $\eta \in L(\mu)$ and $\sigma \in NL(\mu)$ such that $\eta, \sigma, \Phi(\eta)$ and $\Phi(\mu)$ have tips equal to $\text{tip}~\mu$. If $\sigma \subseteq \Phi(\eta)$, then $\sigma \subseteq \Phi(\mu)$.
\end{theorem}
\begin{proof}
	Suppose that $\sigma \nsubseteq \Phi(\mu)$. Then, by Lemma \ref{prp_pro}, there exists a proper $L$-subgroup $\theta$ of $\mu$ such that $\mu = \theta \circ \sigma$. Since $\text{tip}~\mu = \text{tip}~\sigma$, $\text{tip}~\theta = \text{tip}~\mu$. Now, since $\sigma \subseteq \Phi(\eta)$, $\sigma \subseteq \eta \subseteq \mu = \theta \circ \sigma$. Therefore, by Lemma \ref{int_pro},
	\[ \eta = \eta \cap \mu = \eta \cap (\theta \circ \sigma) = (\eta \cap \theta) \circ \sigma. \]
	Again, since $\sigma \subseteq \Phi(\eta)$, by Lemma \ref{prp_pro}, $\eta \cap \theta = \eta$. Therefore, $\sigma \subseteq \eta \subseteq \theta$, which imples that $\mu = \theta \circ \sigma = \theta$,
	which contradicts the assumption that $\theta$ is a proper $L$-subgroup of $\mu$. Hence, we must have $\sigma \subseteq \Phi(\mu)$.
\end{proof}

\section*{Acknowledgements}
The second author of this paper was supported by the Junior Research Fellowship jointly funded by CSIR and UGC, India during the course of development of this paper.

\end{document}